\documentclass{amsart}
\usepackage[utf8]{inputenc}

\usepackage{amsmath,amsthm,amsfonts,enumerate,amssymb,verbatim,graphicx}

\usepackage{caption}

\usepackage{hyperref}
\hypersetup{
colorlinks=true,
    linkcolor=blue,
    urlcolor = blue}

\usepackage{yfonts}

\usepackage{tikz,mathdots,yhmath,cancel,color,siunitx,array,multirow,tabularx,booktabs}

\title{The Generic Bipartite Graphs of Diameter $3$: Their Ages and their almost sure Theories}
\author{Rebecca Coulson}
\date{April 2020}

\newtheorem{lemma}{Lemma}[section]
\newtheorem{corlemma}{Corollary}[lemma]
\newtheorem{obs}{Observation}[section]
\newtheorem{maintheorem}{Theorem}
\newtheorem{thm}{Theorem}[section]
\newtheorem{defn}{Definition}[section]

\newtheoremstyle{noperiod}% name
{3pt}% space above
{3pt}% space below
{\em}% body font
{}% indent amount
{\bf}% theorem head font
{}% punctuation after theorem head
{.2em}% space after theorem head
{}% theorem head spec

\theoremstyle{noperiod}

\newcommand{\firstaclass}{\mathcal{A}^3_{\infty,0,7,8}}
\newcommand{\secondaclass}{\mathcal{A}^3_{\infty,0,7,10}}

\newcommand{\firstGamma}{\Gamma^{3}_{\infty,0,7,8}}
\newcommand{\secondGamma}{\Gamma^{3}_{\infty,0,7,10}}
\newcommand{\Frse}{Fra\"{i}ss\'{e} }
\newcommand{\Th}{$Th$}

\newcommand*\pFqskip{8mu}
\catcode`,\active
\newcommand*\pFq{\begingroup
        \catcode`\,\active
        \def ,{\mskip\pFqskip\relax}%
        \dopFq
}
\catcode`\,12
\def\dopFq#1#2#3#4#5#6{%
        {}_{#1}F_{#2}\bigl[#3,#4;#5;#6\bigr]%
        \endgroup
}

\begin{document}

\maketitle

\begin{abstract}
    In an effort to find more examples of amalgamation classes whose almost sure theories are the same as their generic theories as well as amalgamation classes whose almost sure theories are different from their generic theories, we address our attention to two new cases: the bipartite diameter $3$ metrically homogeneous graphs of generic type. These graphs were identified by Cherlin in \cite{CheCat}, and are determined by certain forbidden configurations. In this paper, we explicitly identify and enumerate their ages, for which we then establish both unlabeled and labeled 0--1 laws. Finally, we show that for one of these bipartite graphs the almost sure theory matches its generic theory, and for the other bipartite graph it does not.
\end{abstract}

\section{Introduction}
It is a well-known result that the class $\mathcal{G}$ of all finite graphs satisfies a \textit{0--1 law}: for any first-order sentence $\phi$ in the language of graphs, $$\lim_{n \rightarrow \infty} \frac{|G \in \mathcal{G}_n \colon G \models \phi|}{|\mathcal{G}_n|} \in \{0,1\},$$ where $\mathcal{G}_n$ is the set of finite graphs on $n$ vertices \cite{Fagin}. Thus we may define the \textit{almost sure theory} of $\mathcal{G}$ to be the collection of sentences whose corresponding limit is $1$. The class $\mathcal{G}$ forms an amalgamation class, and its \Frse limit is the Rado graph. The theory of the Rado graph, the \textit{generic theory} of $\mathcal{G}$, is the same as the almost sure theory of $\mathcal{G}$.

On the other hand, the amalgamation class $\mathcal{G}_{\bigtriangleup}$ of finite triangle-free graphs also satisfies a 0--1 law, but its almost sure theory does \textit{not} match the theory of its \Frse limit, the generic triangle-free graph. Almost every large triangle-free graph is bipartite---that is, it omits odd cycles---while the generic triangle-free graph is not bipartite. This behavior continues for $K_n$-free graphs generally: the class of $K_n$-free graphs for a fixed finite $n$ forms an amalgamation class with a 0--1 law, and its almost sure theory is $(n-1)$-partite, while its generic theory is not \cite{KPR}.

Thus we have the following question.\\

\textbf{Big Question:} When does the almost sure theory of an amalgamation class match the generic theory of its \Frse limit?\\

Aspects and specific instances of this question have been studied (see \cite{Ahlman}, \cite{Hill}, and \cite{Kruck} for some) but the general question remains open.

%This fact follows easily from first principles.
%\begin{fact}
%Let $\mathcal{A}$ be an amalgamation class with a 0--1 law.

%If the almost sure theory of $\mathcal{A}$ matches the generic theory of $\mathcal{A}$, then the generic theory is pseudofinite. 
%\end{fact}

%The converse is unknown---for example, it is a long open question as to whether the generic triangle-free graph is pseudofinite.

We address our question to two other graph \Frse limits determined by forbidden configurations, though here their amalgamation classes consist of finite integer-valued metric spaces. Specifically, we will be examining the two diameter $3$ bipartite metrically homogeneous graphs of generic type.

\textit{Throughout this article, we will use $\mathcal{A}$ to refer to amalgamation classes and $\Gamma$ to refer to \Frse limits.}

\subsection{Metrically homogeneous graphs of generic type}
Homogeneous graphs are graphs in which every finite partial automorphism can be extended to a full automorphism. The finite and countably infinite homogeneous graphs were classified in \cite{Gard,Shee,LaW-HG}. Metrically homogeneous graphs comprise a larger class of graphs---these are graphs in which every finite partial \textit{isometry}, using the path metric, can be extended to a full isometry.

The finite metrically homogeneous graphs were classified by Cameron in \cite{Cam_TRANS}. Cherlin has a tentative catalog of all countable metrically homogeneous graphs \cite{CheCat}. A large class of these are metrically homogeneous graphs of ``generic type." All of the known metrically homogeneous graphs of generic type are themselves \Frse limits, though of finite integer-valued metric spaces which omit certain ``metric triangles" as well as certain technical ``Henson constraints."

By \emph{metric triangles} we mean triples of distances $(d_1,d_2,d_3)$ for which there are three vertices $v_1,v_2,v_3$ where $d(v_1,v_2) = d_1,$ $d(v_1,v_3) = d_2$, and $d(v_2,v_3) = d_3$. The forbidden triangles for a given metrically homogeneous graph of generic type are determined by five parameters $(\delta,K_1,K_2,C,C')$, where
\begin{itemize}
\item $\delta$ is the diameter, i.e. is the max distance between two vertices;
\item $K_1$ is the smallest $k$ for which the metric triangle $(1,k,k)$ is realized;
\item $K_2$ is the largest $k$ for which the metric triangle $(1,k,k)$ is realized;
\item $C_0,C_1$ are the smallest even, odd numbers respectively such that there are no triangles of even/odd perimeter at least $C_0$, $C_1$; $C = \min(C_0,C_1)$, and $C' = \max(C_0,C_1)$.
\end{itemize}
We refer to such a metrically homogeneous graph as $\Gamma^{\delta}_{K_1,K_2,C,C'}$ and its amalgamation class as $\mathcal{A}^{\delta}_{K_1,K_2,C,C'}$

Some metrically homogeneous graphs of generic type have additional \emph{Henson constraints} $\mathcal{S}$ which consist of forbidden spaces only taking distances in $\{1,\delta\}$. The metrically homogeneous graphs we consider in this article have $\mathcal{S} = \emptyset$ and thus we will restrict our attention to forbidden triangles.

We work within an integer-valued metric space in which distances are either $1,2$ or 3; we view these spaces as complete graphs with edges weighted $1,2$ or $3$. We refer to an edge of weight $i$ as an \emph{$i$-edge}. A collection of vertices all of which are pairwise connected via $i$-edges will be called an \emph{$i$-clique}.

The metrically homogeneous graphs of diameter $3$ were classified by Amato, Cherlin, and MacPherson in \cite{ACM-MH3}. We examine here the two bipartite metrically homogeneous graphs of diameter $3$:
\begin{align*}
    \firstGamma && \text{ and } && \secondGamma.
\end{align*}

In Section \ref{sec:expId}, we explicitly identify their ages and enumerate them both up to isomorphism and generally. 
%In Section \ref{sec:enumBs}, we enumerate their ages, now not up to isomorphism. In Section \ref{sec:idGams}, we use this identification to illustrate $\firstGamma$ and $\secondGamma$. In Section \ref{sec:01laws}, we show that $\firstaclass$ and $\secondaclass$ both satisfy 0--1 laws. 
In Section \ref{sec:idGams}, we discuss the global structure of $\firstGamma$ and $\secondGamma$. In Section \ref{sec:01laws}, we establish both unlabeled and labeled 0--1 laws for $\firstaclass$ and $\secondaclass$.
Finally, in Section \ref{sec:asNgen}, we show that the almost sure theory and generic theory diverge for $\firstaclass$, but converge for $\secondaclass$. In Section \ref{sec:furtherwork}, we conclude with some open questions.

\subsection{Results}

In \cite{ACM-MH3}, Amato, Cherlin, and MacPherson classified all the metrically homogeneous graphs of diameter $3$. There are $13$ metrically homogeneous graphs of generic type of diameter $3$.

The only two bipartite metrically homogeneous graphs of generic type of diameter $3$ are $\firstGamma$ and $\secondGamma$. %The triangles in $\firstaclass$ or $\secondaclass$ are apparent from the parameters. 
Our first main results are the explicit classifications and enumerations of all the spaces in $\firstaclass$ and of all the spaces in $\secondaclass$.

\begin{maintheorem}\label{thm:theMSs:firstclass}
Let $A$ be a metric space on $n \geq 3$ vertices in $\firstaclass$. Then $A$ must consist of two $2$-cliques $Q_1,Q_2$ on $k$ vertices and $n-k$ vertices respectively, for any $0 \leq k \leq n-k$, in which each of the vertices from $Q_1$ have either no or one $3$-edge into $Q_2$, and the other connecting edges are $1$-edges. 
For a fixed $k$, there are up to isomorphism exactly $k+1$ possible clique pairs.
This gives us that there are up to isomorphism $$\frac{\big\lfloor \frac{n}{2} \big\rfloor ^2}{2} + \frac{3 \big\lfloor \frac{n}{2} \big\rfloor}{2} + 1$$ metric %same as $$1 + \sum_{i=1}^{\big\lfloor \frac{n}{2} \big\rfloor} i+1$$
spaces on $n$ vertices in $\firstaclass$.

Furthermore, when counting not up to isomorphism, there are
\begin{gather*}
\sum_{k=0}^{\lfloor \frac{n}{2} \rfloor}{ \sum_{j=0}^k {\binom{n}{k} \binom{k}{j} \binom{n-k}{j} j! } }\\
= \sum_{k=0}^{\lfloor \frac{n}{2} \rfloor} (-1)^k \binom{n}{k}U(-k,n-2k+1,-1) %was at one point, but this doesn't really make sense:   \frac{((-n+k-1)!)^2}{k!((-n-1)!)^2}
\end{gather*}
metric spaces on $n$ vertices in $\firstaclass$, where $U$ is the hypergeometric confluent function.
\end{maintheorem}

\begin{maintheorem}\label{thm:theMSs:secondclass}
Let $A$ be a metric space on $n \geq 3$ vertices in $\secondaclass$. Then $A$ must consist of two $2$-cliques $Q_1,Q_2$, on $k$ vertices and $n-k$ vertices respectively, for any $0 \leq k \leq n-k$, in which each of the vertices from $Q_1$ have any configuration of $1$-edges and $3$-edges into $Q_2$. For a fixed $k$, there are up to isomorphism exactly $$\frac{\displaystyle \prod_{i=1}^k n -k+i }{k!} = \frac{\Gamma(1+n)}{\Gamma(1+n-k)k!}$$ possible clique pairs.

In particular, this gives us that there are up to isomorphism $$\sum_{k=0}^{\lfloor \frac{n}{2} \rfloor} \frac{\displaystyle \prod_{i=1}^k n -k+i }{k!} = 2^n \Gamma\Big(n-\Big\lfloor \frac{n}{2} \Big\rfloor\Big) \Gamma\Big(\Big\lfloor \frac{n}{2} \Big\rfloor +2\Big)-\frac{\Gamma(n+1) \hspace{1mm} \pFq{2}{1}{1}{-n+\big\lfloor \frac{n}{2} \big\rfloor+1}{\big\lfloor \frac{n}{2} \big\rfloor +2}{-1}}{\Gamma(n-\big\lfloor \frac{n}{2} \big\rfloor) \Gamma(\big\lfloor \frac{n}{2} \big\rfloor+2)}$$ %was 1 + __ initially, with k beginning at 1
 metric spaces on $n$ vertices in $\secondaclass$, where $_2 F_1$ is the hypergeometric function.
 
 When counting not up to isomorphism, there are
$$\sum_{k=0}^{\lfloor \frac{n}{2} \rfloor}  \binom{n}{k} 2^{k(n-k)}$$ metric spaces on $n$ vertices in $\secondaclass$.
\end{maintheorem}

We note that when we say $A$ must consist of two $2$-cliques, we allow for the possibilities of one of the $2$-cliques being empty or consisting of a single vertex.

After identifying these spaces, we use these classifications to prove that both $\firstaclass$ and $\secondaclass$ satisfy 0--1 laws, and in doing so we identify their almost sure theories. Finally, we compare these almost sure theories to the corresponding generic theories, and we establish that the almost sure theory and the generic theory match for $\firstaclass$ but not for $\secondaclass$. 

Since these classes consist of metric spaces with distances $1,2,$ or $3$, the language we use is finite binary relational, with a relation for each possible distance: $\{R_1,R_2,R_3\}$.

\begin{maintheorem}\label{thm:firstASnGen}
The space $\mathcal{A} = \firstaclass$ satisfies both an unlabeled and a labeled 0--1 law: any first-order sentence $\phi$ in  $\mathcal{L} = \{R_1,R_2,R_3\}$ will satisfy $$ \lim_{n \to \infty} \textgoth{m}_{\mathcal{A},\phi}(n) \in \{0,1\}$$ where $n$ is the number of vertices of the spaces in $\firstaclass$, and $\textgoth{m}_{\mathcal{A},\phi}(n)$ is the proportion of $n$-vertex spaces in $\mathcal{A}=\firstaclass$ which satisfy a fixed $\phi $. This convergence happens whether we count up to isomorphism or not.

Moreover, the almost sure theory of $\firstaclass$ diverges from the generic theory of $\firstGamma$.
\end{maintheorem}

\begin{maintheorem}\label{thm:secondASnGen}
The space $\mathcal{A} = \secondaclass$ satisfies both an unlabeled and a labeled 0--1 law: any first-order sentence $\phi$ in $\mathcal{L} = \{R_1,R_2,R_3\}$ will satisfy $$ \lim_{n \to \infty} \textgoth{m}_{\mathcal{A},\phi}(n) \in \{0,1\}$$ where $n$ is the number of vertices of the spaces in $\secondaclass$, and $\textgoth{m}_{\mathcal{A},\phi}(n)$ is the proportion of $n$-vertex spaces in $\mathcal{A}=\secondaclass$ which satisfy a fixed $\phi$. This convergence happens whether we count up to isomorphism or not.

Moreover, the almost sure theory of $\secondaclass$ agrees with the generic theory of $\secondGamma$.
\end{maintheorem}

\section{\texorpdfstring{The Explicit Identification and Enumerations of $\firstaclass$ and of $\secondaclass$}{expId}}\label{sec:expId}

The driving difference between $\firstaclass$ and $\secondaclass$ is that $\secondaclass$ allows the metric triangle $(2,3,3)$. Thus our process for identifying these two classes will be quite similar, though we will ultimately have the possibility for more $3$-edges and thus more spaces in $\secondaclass$.

\subsection{\texorpdfstring{The explicit identification of $\firstaclass$ and its enumeration up to isomorphism}{expId2}}\label{sec:firstAclass}

We begin by observing the following.
%Since we have excluded any triangle of odd perimeter and any triangle of perimeter more than $6$, we are left with the three possible metric triangles:
\begin{obs}
The only triangles in $\firstaclass$ are
\begin{itemize}
\item $(1,1,2)$;
\item $(1,2,3)$;
\item $(2,2,2)$.
\end{itemize}
\end{obs}

\begin{proof}
This follows immediately from the fact that we have excluded all triangles of odd perimeter and triangles of perimeter more than $6$. Note that $(1,1,3)$ is not a triangle, as it violates the triangle inequality.
\end{proof}

We thus have the following.

\begin{obs}\label{obs:musthavetwo}
Any collection of three vertices must have at least one pairwise distance of $2$.
\end{obs}

%We will see that configurations in which every pairwise distance is $2$ will be central to our examination. We call a collection of vertices in which every pairwise distance is $2$ a \emph{$2$-clique}. We have the following.

Moreover, we also get the following.

\begin{lemma}\label{lemma:2scomeincliques}
Any $2$-edge is part of a $2$-clique.
\end{lemma}

\begin{proof}
An isolated $2$-edge constitutes a $2$-clique, namely a $2$-clique of $2$ vertices.

If there are two adjacent $2$-edges, then the only triangle they can be a part of is $(2,2,2)$. Any larger collection of $k$ connected $2$-edges will have its adjacent pairs all be in the metric triangle $(2,2,2)$; thus this collection must be part of a $2$-clique on $k+1$ vertices.
\end{proof}

We will use now the following terminology.

\begin{defn}
A $2$-clique is \emph{maximal} if there is no strictly larger $2$-clique containing it.
\end{defn}

\begin{lemma}\label{lemma:atMostTwo2cliques}
A space $A \in \firstaclass$ can have at most two maximal $2$-cliques.
\end{lemma}

\begin{proof}
Assume towards a contradiction that $A$ contains $3$ distinct maximal $2$-cliques, which we write as $Q_1,Q_2,Q_3$. By Lemma \ref{lemma:2scomeincliques}, there are no $2$-edges between these cliques, as they would otherwise not be distinct. Thus the edges between them are weighted either $1$ or $3$. However then any triangle $(v_1,v_2,v_3)$ comprising of vertices from $Q_1,Q_2,Q_3$ respectively must have odd perimeter. Thus no such configuration is possible.
\end{proof}

\begin{lemma}\label{lemma:2cliquesPartition}
For any space $A \in \firstaclass$, if $A$ contains two maximal $2$-cliques, then these $2$-cliques partition the vertices of $A$. That is, each vertex of $A$ will be in exactly one $2$-clique.
\end{lemma}

\begin{proof}
Assume that $A$ has two maximal $2$-cliques $Q_1,Q_2$. Assume moreover, towards a contradiction, that there is some vertex $v$ outside of $Q_1$ and $Q_2$. We consider a triangle $(v,q_1,q_2)$ where $q_1,q_2$ are vertices in $Q_1,Q_2$ respectively.

As noted in Observation \ref{obs:musthavetwo}, $(v,q_1,q_2)$ must contain a $2$-edge. This contradicts the assumption that $Q_1,Q_2$ are distinct $2$-cliques.

Thus if there are two distinct $2$-cliques in $A$, then the maximality and the distinctness of these cliques imply that
every vertex in $A$ must be in exactly one of the $2$-cliques.
\end{proof}

\begin{corlemma}\label{cor:oneOrTwo2cliques}
Any metric space $A \in \firstaclass$ on $n \geq 3$ vertices consists of two $2$-cliques, one on $k$ vertices and the other on $n-k$ vertices, for some $0 \leq k \leq n-k$.
\end{corlemma}

\begin{proof}
This follows immediately from Observation \ref{obs:musthavetwo}, and Lemmas \ref{lemma:2scomeincliques}, \ref{lemma:atMostTwo2cliques}, and \ref{lemma:2cliquesPartition}, and from observing once more that $k=0,1$ allows for trivial $2$-cliques.
\end{proof}

\begin{lemma}\label{lemma:havingtwo2cliques}
There are up to isomorphism exactly $k+1$ spaces with $n$ vertices in $\firstaclass$ which contain a $2$-clique $Q_1$ on $k$ vertices and a distinct $2$-clique $Q_2$ on $n-k$ vertices, for $0 \leq k \leq n-k$. Explicitly, each vertex $v$ from $Q_1$ contains either no or one $3$-edge into $Q_2$, and every other edge from $v$ into $Q_2$ is a $1$-edge.
\end{lemma}

\begin{proof}
The triangles in the spaces described will only be $(1,1,2)$, $(1,2,3)$, or $(2,2,2)$. As anywhere from none to all of the $k$-many vertices in $Q_1$ can have a $3$-edge into $Q_2$, we have that there are $k+1$ possible such spaces.
\end{proof}

We now have what we need to prove the first enumeration from Theorem \ref{thm:theMSs:firstclass}.

\begin{thm}\label{thm:unlabeledfirstcount}
For a fixed $n$, there are up to isomorphism $$\frac{\big\lfloor \frac{n}{2} \big\rfloor ^2}{2} + \frac{3 \big\lfloor \frac{n}{2} \big\rfloor}{2} + 1$$ metric %same as $$1 + \sum_{i=1}^{\big\lfloor \frac{n}{2} \big\rfloor} i+1$$
spaces on $n$ vertices in $\firstaclass$.
\end{thm}

\begin{proof}
%We find the metric spaces from Corollary \ref{cor:oneOrTwo2cliques}. 
The number of metric spaces follows Lemma \ref{lemma:havingtwo2cliques}. Finally, we note that 
$$\sum_{k=0}^{\lfloor \frac{n}{2} \rfloor} k +1 = \sum_{k=1}^{\lfloor \frac{n}{2} \rfloor+1} k = \frac{(\big\lfloor \frac{n}{2} \big\rfloor + 2) (\big\lfloor \frac{n}{2} \big\rfloor +1)}{2} = \frac{\big\lfloor \frac{n}{2} \big\rfloor ^2}{2} + \frac{3 \big\lfloor \frac{n}{2} \big\rfloor}{2} + 1,$$ the first term being the sum resulting directly from Lemma \ref{lemma:havingtwo2cliques}.
\end{proof}

%\section{\texorpdfstring{The general enumeration of $\firstaclass$}{enumbclasses}}\label{sec:enumBs}
\subsection{\texorpdfstring{The general enumeration of $\firstaclass$}{enumbclasses}}\label{sec:enumBs}
%\subsection{\texorpdfstring{Enumerating $\firstbclass$}{fbc}}
%We consider now the \emph{labeled} case, in which we have a class $\mathcal{G}$ of metric spaces on $n$ vertices labeled $\{v_1,\cdots,v_n\}$. This in contrast to the case we have been considering, the \emph{unlabeled} case, in which the vertices are not labeled. This second class of graphs is equivalent to the isomorphism types of the first class.

%(The merits of considering one versus the other)

%Once again, we know from Theorem \ref{thm:theMSs:firstclass} the spaces in $\firstaclass$ look like and how many there are. We use this classification and enumeration to enumerate the spaces in $\firstbclass$.

We can directly prove the second enumeration from Theorem \ref{thm:theMSs:firstclass}.

\begin{thm}\label{thm:labeledfirstcount}
There are
\begin{gather*}
\sum_{k=0}^{\lfloor \frac{n}{2} \rfloor}{ \sum_{j=0}^k {\binom{n}{k} \binom{k}{j} \binom{n-k}{j} j! } }\\
= \sum_{k=0}^{\lfloor \frac{n}{2} \rfloor} (-1)^k \binom{n}{k}U(-k,n-2k+1,-1) %was at one point, but this doesn't really make sense:   \frac{((-n+k-1)!)^2}{k!((-n-1)!)^2}
\end{gather*}
metric spaces on $n$ vertices in $\firstaclass$, where $U$ is the hypergeometric confluent function.
\end{thm}

\begin{proof}
Fix some $k$ such that $0 \leq k \leq n-k$. Fix a metric space in $\firstaclass$ consisting of two $2$-cliques $Q_1$,$Q_2$ on $k$ and $n-k$ vertices such that there are $j$ vertices in $Q_1$ with matching vertices in $Q_2$ at distance $3$, for some $0 \leq j \leq k$. There are
$$\binom{n}{k} \binom{k}{j}\binom{n-k}{j}{j!}$$
such spaces in $\firstaclass$ for a fixed $n$, $k$, and $j$. We see this as follows: there are $\binom{n}{k}$ ways to divvy up the $n$ vertices between $Q_1$ and $Q_2$. There are $\binom{k}{j} \binom{n-k}{j}$ ways to choose $j$ pairs in $Q_1$ and $Q_2$ which have a matched clique at distance $3$ in the opposite clique. Finally, there are $j!$ ways to arrange these pairs.

Thus for a fixed $n$, we have 
\begin{gather*}
\sum_{k=0}^{\lfloor \frac{n}{2} \rfloor}{ \sum_{j=0}^k {\binom{n}{k} \binom{k}{j} \binom{n-k}{j} j! } }\\
= \sum_{k=0}^{\lfloor \frac{n}{2} \rfloor} (-1)^k \binom{n}{k}U(-k,n-2k+1,-1) %was at one point, but this doesn't really make sense:   \frac{((-n+k-1)!)^2}{k!((-n-1)!)^2}
\end{gather*}
metric spaces in $\firstaclass$, when not counting generally, and where $U$ is the hypergeometric confluent function.
%Mathematica gives a closed form version of this, but it's messy

We recall for comparison that there are $$ \sum_{k=0}^{\lfloor n/2 \rfloor}{ \sum_{j=0}^k {1}}$$ isomorphism classes of spaces on $n$ vertices in $\firstaclass$.
\end{proof}

Thus we have proved our first main theorem.

\begin{proof}[Proof of Theorem \ref{thm:theMSs:firstclass}]
This follows directly from Lemma \ref{lemma:havingtwo2cliques}  and Theorems \ref{thm:unlabeledfirstcount} and \ref{thm:labeledfirstcount}.
\end{proof}

We provide here for illustration the isomorphism classes of metric spaces on $5$ vertices in $\firstaclass$.

\input{Diagram_firstclass_fiveVs.tex}

\subsection{The explicit identification of \texorpdfstring{$\secondaclass$}{second} and its enumeration up to isomorphism }\label{sec:secondAclass}
The structure of our analysis here closely follows that of $\firstaclass$.

We begin by observing the following.
\begin{obs}
The only triangles in $\secondaclass$ are
\begin{itemize}
    \item $(1,1,2)$;
    \item $(1,2,3)$;
    \item $(2,2,2)$;
    \item $(2,3,3)$.
\end{itemize}
\end{obs}

Thus Observation \ref{obs:musthavetwo} still holds, that is, any collection of three vertices must have at least one pairwise distance of $2$. Lemmas \ref{lemma:2scomeincliques}, \ref{lemma:atMostTwo2cliques}, and \ref{lemma:2cliquesPartition} therefore also hold. Our first difference between $\firstaclass$ and $\secondaclass$ is the following.

\begin{lemma}\label{lemma:oneOrTwo2cliques:second}
Any metric space $A \in \secondaclass$ on $n\geq 3$ vertices consists of two $2$-cliques, one on $k$ vertices and one on $n-k$ vertices, for some $0 \leq k \leq n-k$, with any configuration of $1$-edges and $3$-edges between the cliques.
\end{lemma}

\begin{proof}
Recall that inclusion in $\secondaclass$ is determined by the exclusion of forbidden metric triangles, or equivalently, by only allowing the triangles $(1,1,2), (1,2,3),(2,2,2)$, and $(2,3,3)$.

The only triangles within a $2$-clique will have the type $(2,2,2)$. The only triangles between $2$-cliques will be of type $(1,1,2)$, $(1,2,3)$, or $(2,3,3)$, all of which are allowed. It is not possible for any odd perimeter triangles or the triple $(1,1,3)$ to be embedded, by Observation \ref{obs:musthavetwo}.
\end{proof}

We can now prove the first enumeration from Theorem \ref{thm:theMSs:secondclass}.
%Theorem \ref{thm:theMSs:secondclass}.

\begin{thm}\label{thm:unlabeledsecondcount}
For a fixed $n,k$ there are up to isomorphism 
$$\frac{\displaystyle \prod_{i=1}^k n -k+i }{k!} = \frac{\Gamma(1+n)}{\Gamma(1+n-k)k!}$$
different possible pairs of cliques in a metric space on $n$ vertices in $\secondaclass$.

In particular, this gives us that there are up to isomorphism $$\sum_{k=0}^{\lfloor \frac{n}{2} \rfloor} \frac{\displaystyle \prod_{i=1}^k n -k+i }{k!} = 2^n \Gamma\Big(n-\Big\lfloor \frac{n}{2} \Big\rfloor\Big) \Gamma\Big(\Big\lfloor \frac{n}{2} \Big\rfloor +2\Big)-\frac{\Gamma(n+1) \hspace{1mm} \pFq{2}{1}{1}{-n+\big\lfloor \frac{n}{2} \big\rfloor+1}{\big\lfloor \frac{n}{2} \big\rfloor +2}{-1}}{\Gamma(n-\big\lfloor \frac{n}{2} \big\rfloor) \Gamma(\big\lfloor \frac{n}{2} \big\rfloor+2)}$$ %was 1 + __ initially, with k beginning at 1
 metric spaces on $n$ vertices in $\secondaclass$, where $_2 F_1$ is the hypergeometric function.
\end{thm}

\begin{proof}
The metric spaces on $n$ vertices in $\secondaclass$ are described in Lemma \ref{lemma:oneOrTwo2cliques:second}. All that remains then is to count the spaces found.

Clearly, there is one $2$-clique on $n$ vertices.

For a fixed $k>0$, there are up to isomorphism $\sum_{i=1}^{n-k+1} i$ possible pairs of cliques $Q_1,Q_2$ on $k$ and $n-k$ vertices respectively. We can see this as follows. Each vertex  $v$ in $Q_1$ can have anywhere from $0$ to $n-k$ $3$-edges extending from it into $Q_2$. We view these options as $k$-tuples, in which each position can take values in $\{0,\cdots,n-k\}$. We require that the tuple must be in nondecreasing order, to ensure that we produce spaces which are unique up to isomorphism. There are $$\sum_{i_k=1}^{n-k+1} \sum_{i_{k-1}=1}^{i_k} \cdots \sum_{i_1 = 1}^{i_2} i$$ such spaces, where there are $k$-many sums. This can be simplified to $$\frac{\displaystyle \prod_{i=1}^k n -k+i }{k!}.$$

Summing over every possible $1 \leq k \leq \big\lfloor \frac{n}{2} \big\rfloor$ yields 
\begin{gather*}
\sum_{k=1}^{\lfloor \frac{n}{2} \rfloor} \frac{\displaystyle \prod_{i=1}^k n -k+i }{k!}\\
= \sum_{k=1}^{\lfloor \frac{n}{2} \rfloor} \frac{\Gamma(1+n)}{\Gamma(1+n-k)k!}\\
= 2^n \Gamma\Big(n-\Big\lfloor \frac{n}{2} \Big\rfloor\Big) \Gamma\Big(\Big\lfloor \frac{n}{2} \Big\rfloor +2\Big)-\frac{\Gamma(n+1) \hspace{1mm} \pFq{2}{1}{1}{-n+\big\lfloor \frac{n}{2} \big\rfloor+1}{\big\lfloor \frac{n}{2} \big\rfloor +2}{-1}}{\Gamma(n-\big\lfloor \frac{n}{2} \big\rfloor) \Gamma(\big\lfloor \frac{n}{2} \big\rfloor+2)}-1
\end{gather*}
spaces on $n$ vertices in $\secondaclass$ which contain two distinct $2$-cliques, where $_2 F_1$ is the hypergeometric function.
\end{proof}

\subsection{\texorpdfstring{The general enumeration of $\secondaclass$}{sbc}}

We can now directly prove the second enumeration from Theorem \ref{thm:theMSs:secondclass}.

\begin{thm}\label{thm:labeledsecondcount}
For a fixed $n$, there are
$$\sum_{k=0}^{\lfloor \frac{n}{2} \rfloor}  \binom{n}{k} 2^{k(n-k)}$$ metric spaces on $n$ vertices in $\secondaclass$.
\end{thm}

\begin{proof}
We again begin by fixing some $k$ such that $0 \leq k \leq n-k$. Fix a metric space in $\secondaclass$ consisting of two $2$-cliques $Q_1$, $Q_2$ on $k$ and $n-k$ vertices respectively. Each vertex from $Q_1$ can have anywhere from none to $n-k$ $3$-edges extending from it, and each vertex in $Q_1$ may choose its $3$-edges independently of the others. Thus for a fixed $n$, we have
$$\sum_{k=0}^{\lfloor \frac{n}{2} \rfloor}  \binom{n}{k} 2^{k(n-k)}$$
labeled metric spaces in $\secondaclass$.

Again, for comparison, there are 
$$\sum_{k=0}^{\lfloor \frac{n}{2} \rfloor} \frac{\Gamma(n+1)}{\Gamma(n-k+1)k!}$$ isomorphism classes of metric
spaces on $n$ vertices in $\secondaclass$.
\end{proof}

Thus we have now shown Theorem \ref{thm:theMSs:secondclass}.

\begin{proof}[Proof of Theorem \ref{thm:theMSs:secondclass}]
This follows immediately from Lemma \ref{lemma:oneOrTwo2cliques:second} and Theorems \ref{thm:unlabeledsecondcount} and \ref{thm:labeledsecondcount}.
\end{proof}

We include here for illustration the isomorphism classes of metric spaces on $5$ vertices in $\secondaclass$.

\input{Diagram_secondclass_fiveVs.tex}

\section{\texorpdfstring{Identifying $\firstGamma$ and $\secondGamma$}{idGams}}\label{sec:idGams}

\subsection{\texorpdfstring{Identifying $\firstGamma$}{idFirstGam}}
We know that $\firstGamma$ must embed all of the spaces found in Section \ref{sec:firstAclass}. The fact that it must embed two arbitrarily large $2$-cliques implies that it must contain two countable $2$-cliques. It cannot contain any vertices which are not in one of these two cliques, because if it did, then as with Lemma \ref{lemma:atMostTwo2cliques}, we would have to have a triangle of odd perimeter. Thus, we see that $\firstGamma$ must be bipartite, with the distances within each part being $2$, and distances between parts being either $1$ or $3$. Moreover, one vertex can have at most one $3$-edge extending from it. As any number of vertices can have a $3$-edge extending from it, each part must embed countably many vertices which have a single $3$-edge extending from it. Thus, contained in each part is a countable collection of vertices which have exactly one $3$-edge extending into the other part. To ensure homogeneity, every vertex will have a single $3$-edge extending from it into the other part. %We note that the countable bipartite structured described embeds, as an induced subgraph, arbitrarily many vertices with only $1$-edges extending to the opposite part.

Thus it must be that $\firstGamma$ consists of two parts, where every vertex is distance $2$ from the vertices in its part, distance $3$ from exactly one vertex in the opposite part, and distance $1$ from every other vertex in the opposite part.

Without loss of generality, $\firstGamma$ resembles the figure below. These two figures are equivalent; the graph can be obtained from the metric space by erasing distances $2$ and $3$. Similarly, the metric space can be obtained by the graph by assigning the path metric. %, with distance $1$ being represented by a black solid line, distance $2$ by a grey solid line, and distance $3$ by a dotted black line.
This graph is the bipartite complement of a matching between infinite sets.\\

\input{Diagram_firstGamma.tex}

\subsection{\texorpdfstring{Identifying $\secondGamma$}{inSecondGam}}
Again, we know that $\secondGamma$ must embed all of the spaces found in Section \ref{sec:secondAclass}. The fact that it must embed two arbitrarily large $2$-cliques implies that it must contain two countable $2$-cliques. It cannot contain any vertices which are not in one of these two cliques. Thus, we see that $\secondGamma$ must be bipartite, with the distances within each part being $2$. 

As with $\firstGamma$, the distances between the parts can either be $1$ or $3$. By our classification, we know that each part of $\secondGamma$ must embed countably many vertices which have countably many $1$-edges and coutably many $3$-edges extending into the opposite part.
%Thus each part must have countably many vertices with countably many $1$-edges extending into the opposite part, and countably many $3$-edges extending into the opposite part. 
To ensure homogeneity, every vertex must have countably many $1$-edges and countably many $3$-edges extending into the other part.%, for if any vertices did not do this, then $\secondGamma$ would not be homogeneous. 

Thus we have identified $\secondGamma$ up to isomorphism; see Figure \ref{fig:secondGamma} below. Again, these two figures are equivalent; the graph can be obtained from the metric space by erasing distances $2$ and $3$, and the metric space can be obtained by the graph by assigning the path metric. This graph is known as the generic bipartite graph of diameter $3$.

\input{Diagram_secondGamma.tex}

\section{\texorpdfstring{Establishing 0--1 laws}{est01Take2}}\label{sec:01laws}

\subsection{\texorpdfstring{Establishing a 0--1 law for $\firstaclass$}{first01}}

We will show that sentences of the following forms sentences axiomatize the almost sure theory of $\firstaclass$. 

\begin{enumerate}
    \item\label{item:possibleDistances} $\forall u \forall v (u \neq v \implies (d(u,v) = 1 \vee d(u,v) = 2 \vee d(u,v) = 3))$ %distances exist, and are 1, 2, or 3
    
    This first sentence says that distance between any two vertices is always defined, and must be $1$, $2$, or $3$.
    %\item\label{item:twosAreTransitive} $\forall u_1,u_2,u_3$ $$d(u_1,u_2) = 2 \wedge d(u_2,u_3) = 2 \implies d(u_1,u_3) = 2.$$%implied by no more than two parts
    \item\label{item:noAdjacentThrees} $\forall u \neg \exists v_1, v_2$ such that $$ v_1 \neq v_2 \hspace{1mm} \wedge \hspace{1mm} d(u,v_1) = 3 \hspace{1mm} \wedge \hspace{1mm} d(u,v_2) = 3 $$
    
    This sentence says that no triple of vertices will have two $3$-edges
    \item\label{item:exactlyTwoParts} $\forall u_1,u_2,u_3 \hspace{1cm}\bigwedge_{i,j} u_i \neq u_j \implies$
    \begin{multline*}
    d(u_1,u_2) = 2 \hspace{1mm} \wedge \hspace{1mm} d(u_1,u_3) = 2 \hspace{1mm} \wedge \hspace{1mm} d(u_2,u_3) = 2 \bigvee
d(u_1,u_2) = 2 \hspace{1mm} \wedge \hspace{1mm} d(u_1,u_3) \neq 2 \hspace{1mm} \wedge \hspace{1mm} d(u_2,u_3) \neq 2\\ \bigvee d(u_1,u_3) = 2 \hspace{1mm} \wedge \hspace{1mm} d(u_1,u_2) \neq 2 \hspace{1mm} \wedge \hspace{1mm} d(u_2,u_3) \neq 2 \bigvee d(u_2,u_3) = 2 \hspace{1mm} \wedge \hspace{1mm} d(u_1,u_2) \neq 2 \hspace{1mm} \wedge \hspace{1mm} d(u_1,u_3) \neq 2
    \end{multline*}%exactlytwoparts
    
    This sentence says that every triple of vertices must have either three $2$-edges or exactly one $2$-edge
    \item\label{item:threePairs} $\exists u_1, \cdots, u_p,$ $\exists v_1,\cdots,v_p$ such that
    $$\bigwedge_{i \neq j} u_i \neq u_j  \bigwedge_{i \neq j} v_i \neq v_j \bigwedge_i d(u_i,v_i) = 3$$ %infinitely many 3-pairs
    
    This sentence says that there are at least $p$ pairs of vertices on opposite parts which are distance $3$ from each other.
    \item\label{item:verticesWithOnlyOnesOver} $\exists u_1,\cdots,u_p$,$\exists v_1,\cdots,v_p$ $\forall w$ such that
    \begin{multline*}
    \bigwedge_{i \neq j} u_i \neq u_j \bigwedge_{i \neq j} v_i \neq v_j \bigwedge_{i,j} u_i \neq v_j \\ \bigwedge_i (\neg d(u_i,w) = 2 \implies d(u_i,w) = 1) \bigwedge_i (\neg d(v_i,w) = 2 \implies d(v_i,w) = 1) 
    \end{multline*}
    %infinitely many with only 1s going to the other part
    
    This sentence says that there are at least $p$ vertices in each part which only have $1$-edges into the opposite part.
\end{enumerate}

For sentence types \ref{item:threePairs}, \ref{item:verticesWithOnlyOnesOver} we vary $p$ over $\mathbb{N}$. Thus \textbf{\textit{we define $\Phi$ to be the collection of all such sentences for every $p \in \mathbb{N}$.}}\\

We will show that the asymptotic probability of isomorphism classes in $\firstaclass$ satisfying any one $\phi \in \Phi$ will go to $1$. We will then introduce a countable structure $\Gamma_{as}$ which satisfies $\Phi$. Finally, we show that $\Phi$ is $\aleph_0$-categorical, that is, given any countable structure $\Gamma$ in the same language where $\Gamma \models \Phi$, we have that $\Gamma \simeq \Gamma_{as}$. This gives us that $\Phi$ is complete. We deduce then that $\firstaclass$ satisfies an labeled 0--1 law. Finally, we make the observation that $\Th(\Gamma_{as}) \neq \Th(\firstGamma)$.

\begin{lemma}\label{lemma:firstclass01}
Define $\textgoth{m}_{\mathcal{A},\phi}(n)$ to be the proportion of isomorphism classes of $n$-vertex spaces in $\mathcal{A}=\firstaclass$ which satisfy a fixed $\phi \in \Phi$.

Then $$\lim_{n \to \infty} \textgoth{m}_{\mathcal{A},\phi}(n) = 1$$ for every $\phi \in \Phi$.
\end{lemma}

\begin{proof}
We address in order the sentence types in $\Phi$.

Type \ref{item:possibleDistances}: By definition, the spaces in $\firstaclass$ have all pairwise distances defined, and each distance is either $1$, $2$, or $3$. Thus $\textgoth{m}_{\mathcal{A},\phi}(n) = 1$ for every $n$.\\

Type \ref{item:noAdjacentThrees}: If this sentence were violated, then there would be a metric triangle $(3,3,x)$. No such metric triangle is allowed, and thus $\textgoth{m}_{\mathcal{A},\phi}(n) = 1$ for every $n$.\\

Type \ref{item:exactlyTwoParts}: All the spaces in $\firstaclass$ either have one part or two parts; thus $\textgoth{m}_{\mathcal{A},\phi}(n) = 1$ for every $n$.\\

Type \ref{item:threePairs}: Call one such sentence $\phi_p$. For a fixed $p,n$ where $p \leq \big\lfloor \frac{n}{2} \big\rfloor$, up to isomorphism the number of spaces in $\firstaclass$ on $n$ vertices which satisfy $\phi_p$ is 
$$\sum_{k \geq p}^{\big\lfloor \frac{n}{2} \big\rfloor} k - p +1.$$ Thus the proportion isomorphism classes of spaces in $\firstaclass$ on $n$ vertices which satisfy $\phi_p$ is 
$$ \frac{\sum\limits_{k \geq p}^{\big\lfloor \frac{n}{2} \big\rfloor} k - p +1}{\sum\limits_{i=1}^{\big\lfloor \frac{n}{2} \big\rfloor +1} i}.$$
This fraction goes to $1$ as $p$ stays fixed and $n$ goes to infinity.\\

Type \ref{item:verticesWithOnlyOnesOver}: For a fixed $p,n$ where $p \leq \big\lfloor \frac{n}{2} \big\rfloor$, the number of isomorphism classes of spaces in $\firstaclass$ on $n$ vertices which satisfy $\phi_p$ are again
$$\sum_{k \geq p}^{\big\lfloor \frac{n}{2} \big\rfloor} k - p +1.$$ Thus again the proportion of isomorphism classes of spaces in $\firstaclass$ on $n$ vertices which satisfy $\phi_p$ is 
$$\frac{\sum\limits_{k \geq p}^{\big\lfloor \frac{n}{2} \big\rfloor} k - p +1}{\sum\limits_{i=1}^{\big\lfloor \frac{n}{2} \big\rfloor +1} i}$$
and this fraction goes to $1$ as $p$ stays fixed and $n$ goes to infinity.
\end{proof}

We now construct the limiting structure $\Gamma_{as}$, and verify that $\Gamma_{as} \models \phi$ for every $\phi \in \Phi$.

\begin{defn}\label{defn:firstAS}
Let $\Gamma_{as}$ be a countable metric space which satisfies the following.
\begin{itemize}
    \item For every pair of points $u \neq v$, we have that $d(u,v) \in \{1,2,3\}$.
    \item $\Gamma_{as}$ is bipartite.
    \item The only distance within a part is $2$.
    \item The distances between parts are either $1$ or $3$.
    \item No vertex can have more than one other vertex at distance $3$.
    \item There are countably many points in a part which have exactly one point in the other part at distance $3$.    
    \item There are countably many points in a part which only have distance $1$ between it and the points in the other part.%\item Every point must satisfy one of the previous two items. That is, every point in a part must either have either exactly one or no point at distance $3$ in the other part, and all other points in the other part are at distance $1$.
\end{itemize}
\end{defn}

\begin{lemma}\label{obs:AS1unique}
There is up to isomorphism exactly one countable metric space which satisfies the description in Definition \ref{defn:firstAS}.
\end{lemma}

We will need to carefully build an mapping which preserves the unique $3$-edge coming from some of the vertices.

\begin{proof}
Let $\Gamma_1$ and $\Gamma_2$ be two countable metric spaces which satisfy Definition \ref{defn:firstAS}. We construct an isomorphism between $\Gamma_1$ and $\Gamma_2$.

Let $(V_1)_3$ be the countable set of vertices from one part of $\Gamma_1$ which have $3$-edges. Let $(V_1)_3'$ be an analogous countable set of vertices in $\Gamma_2$, that is, the countable set of vertices from one part of $\Gamma_2$ which have $3$-edges. We note that it does not matter which parts of $\Gamma_1,\Gamma_2$ have been chosen, just that the vertices come from a single part. 

Let $f_0$ be any bijection between $(V_1)_3$ and $(V_1)_3'$. We create an extension $f_1$ of $f_0$ as follows: for every $u \in S_3$, let $v \in \Gamma_1$ be the unique vertex such that $d(u,v) = 3$, and for every $u' \in S_3'$ let $v' \in \Gamma_2$ be the unique vertex such that $d(u',v') = 3$. Then $f_1(v) = v'$ for every such $v \in \Gamma_1$ and $v' \in \Gamma_2$.

Now let $(V_1)_1$ be the countable set of vertices from the same part of $\Gamma_1$ which do not have vertices at distance $3$, and let $(V_1)_1'$ be an analogous countable set from the same part of $\Gamma_2$. Let $f_2$ be any bijection between $(V_1)_1$ and $(V_1)_1'$. Moreover let $(V_2)_1$ be the countable set of vertices from the other part of $\Gamma_1$ which do not have vertices at distance $3$, and $(V_2)_1'$ be the countable set of vertices from the other part of $\Gamma_2$ which do not have vertices at distance $3$. Let $f_3$ be any bijection between $(V_2)_1$ and $(V_2)_1'$.

Define $f = \bigcup_{i \leq 3} f_i$. 

Since every vertex in $\Gamma_1, \Gamma_2$ either does or does not have a vertex at distance $3$, $f$ is a map from all of $\Gamma_1$ to all of $\Gamma_2$. Moreover, since $f$ is the union of disjoint bijections, it is itself a bijection.

We now verify that $f$ preserves distances. Take $u,w \in \Gamma_1$. By the definition of $f_1$, $d(u,w) = 3 \Longleftrightarrow d(f(u),f(w)) = 3$. 

If $d(u,w) = 1$, then $u$ and $w$ are in opposite parts. 
%Assume without loss of generality that $S_3$ and $S_{1,1}$ are from the same part of $\Gamma_1$. 
If $u \in (V_1)_3$ and $w \in (V_2)_1$, then $f(u) \in (V_1)_3'$ and $f(w) \in (V_2)_1'$ and so $d(f(u),f(w)) = 1$. If $u \in (V_1)_3$ and $w \in V_2 \setminus (V_2)_1$ (that is, $w$ in the opposite part but not in $(V_2)_1$), then $d(f(u),f(w)) \in \{1,3\}$. Since $d(u,w) = 1$ implies that $w \neq v$, we have that $d(f(u),f(w)) = 1$. Our argument here is symmetric, so we have that $d(u,v) = 1 \Longleftrightarrow d(f(u),f(v)) = 1$.

As we already showed that $d(u,v) = 3 \Longleftrightarrow d(f(u),f(v)) = 3$, and $f$ is a bijection, we also have that $d(u,v) = 2 \Longleftrightarrow d(f(u),f(v)) = 2$.
\end{proof}

Below are two representations of $\Gamma_{as}$---one as a metric space, and one as a graph. Distance $1$ is represented by a solid black line, distance $2$ by a solid grey line, and distance $3$ by a dotted black line.

\input{Diag_firstGam_as.tex}

Note that $\Gamma_{as}$ is not metrically homogeneous; it is not even vertex-transitive.\\

%We also observe the following.
%
%\begin{obs}\label{as1modelsPhi}
%$\Gamma_{as} \models \phi$ for every $\phi %\in \Phi$.
%\end{obs}

Thus, we have shown the first part of Theorem \ref{thm:firstASnGen}.

\begin{thm}\label{thm:firstUnlabeled01}
The amalgamation class $\firstaclass$ satisfies an unlabeled 0--1 law.
\end{thm}

\begin{proof}
Lemma \ref{lemma:firstclass01} established a set of axioms which asymptotically go to $0$ among the isomorphism classes of spaces in $\firstaclass$. Lemma \ref{obs:AS1unique} and the observation that $\secondGamma$ satisfies $\Phi$ imply that this these axioms axiomatize a complete theory.
\end{proof}

%\begin{proof}
%We see immediately upon inspection that $\Gamma_{as}$ satisfies every type of $\phi$.
%\end{proof}

\subsection{\texorpdfstring{Establishing a labeled 0--1 law for $\firstaclass$}{first01}}

It suffices to check that the proportion of metric spaces of $\firstaclass$, now no longer counting up to isomorphism, which satisfy the axiomatizing set of sentences $\Phi$ still asymptotically goes to $1$.
%
%We briefly show the sentences which axiomatize the almost sure theory of $\firstaclass$ also axiomatize the almost sure theory of $\firstbclass$. 
%
Recall that the following sentences axiomatize $\firstaclass$ when we vary $p$ over $\mathbb{N}$. 

\begin{enumerate}
    \item\label{item:possibleDistances:B1} $\forall u \forall v (u \neq v \implies (d(u,v) = 1 \vee d(u,v) = 2 \vee d(u,v) = 3))$ 
    %distances exist, and are 1, 2, or 3
    %\item\label{item:twosAreTransitive} $\forall u_1,u_2,u_3$ $$d(u_1,u_2) = 2 \wedge d(u_2,u_3) = 2 \implies d(u_1,u_3) = 2.$$%implied by no more than two parts
    \item\label{item:noAdjacentThrees:B1} $\forall u \neg \exists v_1, v_2$ such that $$ v_1 \neq v_2 \hspace{1mm} \wedge \hspace{1mm} d(u,v_1) = 3 \hspace{1mm} \wedge \hspace{1mm} d(u,v_2) = 3 $$
    \item\label{item:exactlyTwoParts:B1} $\forall u_1,u_2,u_3 \hspace{1cm}\bigwedge_{i,j} u_i \neq u_j \implies$
    \begin{multline*}
    d(u_1,u_2) = 2 \hspace{1mm} \wedge \hspace{1mm} d(u_1,u_3) = 2 \hspace{1mm} \wedge \hspace{1mm} d(u_2,u_3) = 2 \bigvee
d(u_1,u_2) = 2 \hspace{1mm} \wedge \hspace{1mm} d(u_1,u_3) \neq 2 \hspace{1mm} \wedge \hspace{1mm} d(u_2,u_3) \neq 2\\ \bigvee d(u_1,u_3) = 2 \hspace{1mm} \wedge \hspace{1mm} d(u_1,u_2) \neq 2 \hspace{1mm} \wedge \hspace{1mm} d(u_2,u_3) \neq 2 \bigvee d(u_2,u_3) = 2 \hspace{1mm} \wedge \hspace{1mm} d(u_1,u_2) \neq 2 \hspace{1mm} \wedge \hspace{1mm} d(u_1,u_3) \neq 2
    \end{multline*}%exactlytwoparts
    \item\label{item:threePairs:B1} $\exists u_1, \cdots, u_p,$ $\exists v_1,\cdots,v_p$ such that
    $$\bigwedge_{i \neq j} u_i \neq u_j  \bigwedge_{i \neq j} v_i \neq v_j \bigwedge_i d(u_i,v_i) = 3$$ %infinitely many 3-pairs
    \item\label{item:verticesWithOnlyOnesOver:B1} $\exists u_1,\cdots,u_p$ $\forall v$ such that
    $$\bigwedge_{i \neq j} u_i \neq u_j \bigwedge_i (\neg d(u_i,v) = 2 \implies d(u_i,v) = 1)$$ %infinitely many with only 1s going to the other part
\end{enumerate}

We once more define $\Phi$ to be the collection of all such sentences for every $p \in \mathbb{N}$.

\begin{lemma}\label{lemma:LabeledFirstPhi01}
Let $\textgoth{m}_{\mathcal{A},\phi}(n)$ to be the proportion of $n$-vertex spaces in $\mathcal{A} = \firstaclass$ which satisfy a fixed $\phi \in \Phi$.
Then $$\lim_{n\rightarrow \infty} \textgoth{m}_{\mathcal{A},\phi}(n)=1$$ for every $\phi \in \Phi$.
\end{lemma}

\begin{proof}
Sentences \ref{item:possibleDistances:B1},  \ref{item:noAdjacentThrees:B1}, and \ref{item:exactlyTwoParts:B1} again hold for every space in $\firstaclass$, making the proportion constantly $1$.

For a fixed $n,p$, there are $$\sum_{k=0}^{\lfloor \frac{n}{2} \rfloor} \sum_{j \geq p}^{k} \binom{k}{j} \binom{n-k}{j}j!$$ spaces on $n$ vertices in $\firstaclass$ which satisfy sentence \ref{item:threePairs:B1}. As $p$ is fixed, we have that $$\lim_{n \rightarrow \infty} \frac{\sum\limits_{k=0}^{\lfloor \frac{n}{2} \rfloor} \sum\limits_{j \geq p}^{k} \binom{k}{j} \binom{n-k}{j}j!}{\sum\limits_{k=0}^{\lfloor \frac{n}{2} \rfloor} \sum\limits_{j =0}^{k} \binom{k}{j} \binom{n-k}{j}j!} = 1.$$

Similarly, for a fixed $n,p$ there are 
$$\sum_{k=0}^{\lfloor \frac{n}{2} \rfloor} \sum_{j =0}^{k-p} \binom{k}{j} \binom{n-k}{j}j!$$ spaces on $n$ vertices in $\firstaclass$ which satisfy sentence \ref{item:verticesWithOnlyOnesOver:B1}. As $p$ is fixed, we have that 
$$\lim_{n \rightarrow \infty} \frac{\sum\limits_{k=0}^{\lfloor \frac{n}{2} \rfloor} \sum\limits_{j =0}^{k-p} \binom{k}{j} \binom{n-k}{j}j!}{\sum\limits_{k=0}^{\lfloor \frac{n}{2} \rfloor} \sum\limits_{j =0}^{k} \binom{k}{j} \binom{n-k}{j}j!} = 1.$$
\end{proof}

Thus we have established another part of Theorem \ref{thm:firstASnGen}.

\begin{lemma}\label{lemma:firstLabeled01}
The amalgamation class $\firstaclass$ satisfies a labeled 0--1 law.
\end{lemma}

\begin{proof}
As with the proof of Theorem \ref{thm:unlabeledfirstcount}, this follows from Lemmas \ref{obs:AS1unique} and \ref{lemma:LabeledFirstPhi01}.% and Observation \ref{as1modelsPhi}.
\end{proof}

\subsection{\texorpdfstring{Establishing an unlabeled 0--1 law for $\secondaclass$}{second01}}

We will show that sentences of the following forms axiomatize the almost sure theory of $\secondaclass$.

\begin{enumerate} 
        \item\label{item:all12or3} $\forall u \forall v (u \neq v \implies (d(u,v) = 1 \vee d(u,v) = 2 \vee d(u,v) = 3))$ 
        
        This sentence says that the only possible distances are $1,2$, and $3$.
    \item\label{item:exactlyTwoPartsAgain} $\forall u_1,u_2,u_3 \hspace{1cm}\bigwedge_{i,j} u_i \neq u_j \implies$
    \begin{multline*}
    d(u_1,u_2) = 2 \hspace{1mm} \wedge \hspace{1mm} d(u_1,u_3) = 2 \hspace{1mm} \wedge \hspace{1mm} d(u_2,u_3) = 2 \bigvee
d(u_1,u_2) = 2 \hspace{1mm} \wedge \hspace{1mm} d(u_1,u_3) \neq 2 \hspace{1mm} \wedge \hspace{1mm} d(u_2,u_3) \neq 2\\ \bigvee d(u_1,u_3) = 2 \hspace{1mm} \wedge \hspace{1mm} d(u_1,u_2) \neq 2 \hspace{1mm} \wedge \hspace{1mm} d(u_2,u_3) \neq 2 \bigvee d(u_2,u_3) = 2 \hspace{1mm} \wedge \hspace{1mm} d(u_1,u_2) \neq 2 \hspace{1mm} \wedge \hspace{1mm} d(u_1,u_3) \neq 2
    \end{multline*}
    
    This sentence says that every triple of vertices must have either three $2$-edges or exactly one $2$-edge
    \item\label{item:infinitelyManyOnesForEach}
    $\forall u \exists v_1,\cdots,v_p$ such that
    $$\bigwedge_{i \neq j} v_i \neq v_j \bigwedge_j d(u,v_i) = 1$$
    
    This sentence says that every vertex has at least $p$ many $1$-edges
    \item\label{item:infinitelyManyThreesForEach}
     $\forall u \exists v_1,\cdots,v_p$ such that
     $$\bigwedge_{i \neq j} v_i \neq v_j \bigwedge_i d(u,v_i) = 3$$
     
     This sentence says that every vertex has at least $p$ many $3$-edges.
\end{enumerate}

\textit{Again, define $\Phi$ to be the collection of all such sentences for every $p \in \mathbb{N}$.}

We have the following.

\begin{lemma}\label{lemma:secondclass01}
Define $\textgoth{m}_{\mathcal{A},\phi}(n)$ to be the proportion of isomorphism classes of $n$-vertex spaces in $\mathcal{A}=\secondaclass$ which satisfy a fixed $\phi \in \Phi$.

Then $$\lim_{n \to \infty} \textgoth{m}_{\mathcal{A},\phi}(n) = 1$$ for every $\phi \in \Phi$.
\end{lemma}

\begin{proof}
We address in order the sentence types in $\Phi$.

Type \ref{item:all12or3} follows as it did with $\firstaclass$.\\

Type \ref{item:exactlyTwoPartsAgain}: Again, since every space in $\secondaclass$ has either one or two parts, we have that $\textgoth{m}_{\mathcal{A},\phi}(n) = 1$ for all $n$.\\

Type \ref{item:infinitelyManyOnesForEach}: 
Since $k = 0$ corresponds to a vanishingly small percentage of the isomorphism classes of spaces in $\secondaclass$, we consider the cases where $k \geq 1$.

Of the $$\sum_{k \geq 1}^{\big\lfloor \frac{n}{2} \big\rfloor}  \sum_{i_k=1}^{n-k+1} \sum_{i_{k-1}=1}^{i_k} \cdots \sum_{i_1 = 1}^{i_2} i$$ many isomorphism classes of spaces consisting of two $2$-cliques, $$\sum_{k \geq p}^{\big\lfloor \frac{n}{2} \big\rfloor}  \sum_{i_k=1}^{n-k+1-p} \sum_{i_{k-1}=1}^{i_k} \cdots \sum_{i_1 = 1}^{i_2} i$$ of them will satisfy $\phi$.

Thus for a fixed $p$, the proportion isomorphism classes of spaces in $\secondaclass$ which satisfies $\phi$ goes to $1$ as $n$ goes to infinity.\\ 

Type \ref{item:infinitelyManyThreesForEach}: Since $1$ and $3$ are equivalent in $\secondaclass$ (that is, they appear symmetrically in the allowed triangles), this is equivalent to the fact that the proportion of the previous type going to $1$.
\end{proof}

%\begin{obs}\label{obs:secondGammaSatisfiesAS}
%$\secondGamma \models \Phi$.
%\end{obs}

%\begin{proof}
%This follows immediately upon inspection.
%\end{proof}

\begin{lemma}\label{lemma:secondASisCat}
Let $\Gamma_1, \Gamma_2$ be two countable metric spaces with distances in $\{1,2,3\}$ for which $\Gamma_1,\Gamma_2 \models \phi$ for every $\phi \in \Phi$. Then $\Gamma_1 \simeq \Gamma_2$.
\end{lemma}

A standard back-and-forth argument works here.

\begin{proof}
The following extension axiom will prove particularly useful:\\

  $\psi_{q,p,r}:= \forall x_1,\cdots,x_q \forall y_1, \cdots,y_p \forall z_1, \cdots, z_q$
    \begin{multline}\label{eqn:extAx2}
    \bigwedge_{i \neq j} d(x_i,x_j) = 2 \bigwedge_{i \neq j} d(y_i,y_j) = 2 \bigwedge_{i \neq j} d(z_i,z_j) = 2\\
    \bigwedge_{i,j} d(x_i,y_j) \neq 2 \bigwedge_{i,j} d(x_i,z_j) \neq 2 \bigwedge_{i,j} d(y_i,z_j) \neq 2 \\ 
    \implies \exists w \ni \bigwedge_i d(w,x_i) = 1 \bigwedge_i d(w,y_i) = 2 \bigwedge_i d(w,z_i) = 3.\tag{*}
\end{multline}

This sentence is implied by sentences $\phi$ of types \ref{item:exactlyTwoPartsAgain}, \ref{item:infinitelyManyOnesForEach}, and \ref{item:infinitelyManyThreesForEach}. Thus $\Gamma_1,\Gamma_2 \models \psi_{q,p,r}$ for every $q,p,r \in \mathbb{N}$.\\

We use a straightforward back-and-forth argument to show that $\Gamma_1 \simeq \Gamma_2$.

We construct an isomorphism $f: \Gamma_1 \rightarrow \Gamma_2$. Let $\{a_0,a_1,\cdots\}$ be a list of the points in $\Gamma_1$ and let $\{b_0,b_1,\cdots\}$ be a list of the points in $\Gamma_2$.\\

Let $f_0 = \emptyset$.

At stage $s +1 = 2i+1$, we check if $a_i$ is in the domain of $f_s$. If it is, then we set $f_{s+1} = f_s$. If it is not, then let $\alpha_1,\cdots,\alpha_m$ list the points in the domain of $f_s$. Define $X := \{j \leq m \colon d(\alpha_j,a_i) = 1\}$, $Y := \{j \leq m \colon d(\alpha_j,a_i) = 2\}$, and $Z := \{j \leq m \colon d(\alpha_j,a_i) = 3\}$. 

By Axiom \ref{eqn:extAx2}, we can find $b \in \{b_0,b_1,\cdots\} \setminus Im(f_s)$ such that $d(f_s(\alpha_j),b) = 1$ for $j \in X$, $d(f_s(\alpha_j),b) = 2$ for $j \in Y$, and $d(f_s(\alpha_j),b) = 3$ for $j \in Z$. Define $f_{s+1} = f_s \cup \{(a_i,b)\}$.

At stage $s+1 = 2i+2$, we check that $b_i$ is in the image of $f_{s+1}$. The argument here is analogous to that of $s+1 = 2i+1$.

Let $f = \bigcup f_s$. We see that $f$ is an isomorphism, as we have ensured that it is both one-to-one and onto and preserves distances.

Thus $\Gamma_1 \simeq \Gamma_2$, and therefore $\Phi$ is $\aleph_0$-categorical.
\end{proof}

We have now the first part of Theorem \ref{thm:secondASnGen}.
\begin{thm}\label{thm:UnlabeledSecond01}
The amalgamation class $\secondaclass$ satisfies an unlabeled 0--1 law.
\end{thm}

\begin{proof}
Lemma \ref{lemma:secondclass01} established a set of axioms for which the proportion of isomorphism classes of spaces in $\secondaclass$ which satisfy them asymptotically go to $1$. The observation that $\secondGamma \models \Phi$ together with Lemma \ref{lemma:secondASisCat} show that this set axiomatizes a complete theory. Thus, $\secondaclass$ satisfies an unlabeled 0--1 law.
\end{proof}

\subsection{\texorpdfstring{Establishing a labeled 0--1 law for $\secondaclass$}{second01}}

The method we used in Lemma \ref{lemma:LabeledFirstPhi01} is difficult to apply here; the combinatorics involved in counting the number of spaces in $\secondaclass$ which satisfy some of the sentences in $\Phi$ becomes complicated.

Instead, we prove asymptotically all the metric spaces in $\secondaclass$ are \textit{asymmetric}, i.e. their only automorphism is the identity automorphism. We use ideas from the proof found in \cite{NesConstraint} that almost all graphs are asymmetric. We note that such an argument would not have been possible to establish a labeled 0--1 law for $\firstaclass$---not to mention inefficient, as the combinatorial argument requires fewer steps---since those spaces cannot be constructed probabilistically.

\begin{lemma}
Almost all metric spaces in $\secondaclass$ are asymmetric.
\end{lemma}

\begin{proof}
%(Give game plan)
%
Our argument uses a basic probabilistic argument. We need therefore to define the notion of a \textit{random} metric space in $\secondaclass$. 

Assume a collection $V$ of $n$ vertices is broken into two parts, $V_1$ and $V_2$. The distances among parts must be $2$. Assign distances between parts by assigning distance $1$ and distance $3$ with equal probability. That is, for $v_1 \in V_1$ and $v_2 \in V_2$, we say
$$\Pr[d(v_1,v_2) = 1] = \frac{1}{2} \hspace{1cm} \text{and} \hspace{1cm} \Pr[d(v_1,v_2) = 3] = \frac{1}{2}.$$\\

We will estimate the number of metric spaces in $\secondaclass$ which would be fixed by a given nontrivial permutation, and we will also estimate the number of nonidentity permutations which are automorphisms of a random metric space in $\secondaclass$. The product of these two provides an upper bound on the number of metric spaces in $\secondaclass$ which admit a nontrivial automorphism, and from there we can calculate the probability that a metric space in $\secondaclass$ admits such an automorphism. We will ultimately show that this probability goes to $0$.\\ 

Let $\phi: V \rightarrow V$ be a nonidentity permutation, and let $A$ be a random metric space in $\secondaclass$ with parts $V_1,V_2$. We fix $x \in V$ such that $\phi(x) = y$, where $y \neq x$. We define $V'$ to be the edges between parts, that is, $V' = \big\{ \{v_1 , v_2\}\big\}_{v_1 \in V_1, v_2 \in V_2}$ and define $\phi':V' \rightarrow V'$ by $\phi'(\{v_1,v_2\}) = \{\phi(v_1),\phi(v_2)\}.$ Observe that $\phi'$ is a permutation.

%Say here what we're doing
We examine first the case $k = \frac{n}{2} (1 - \epsilon)$, $\epsilon \neq 0$. In particular, this implies that $x$ and $y$ are on the same side. Given a vertex $v$ on the opposite side of $x$ and $y$, there are four possibilities for the pair of distances $\{d(x,v),d(y,v)\}$: 
\begin{itemize}
    \item $d(x,v) = d(y,v) = 1$;
    \item $d(x,v) = d(y,v) = 3$;
    \item $d(x,v) = 1$ and $d(y,v) = 3$;
    \item $d(x,v) = 3$ and $d(y,v) = 1$.
\end{itemize}
As each of these possibilities are equally likely, and as there are at least $\frac{n}{2}(1 - \epsilon)$ vertices on the side opposite $x$ and $y$, we have that there are at least $\frac{n}{4}(1 - \epsilon)$ vertices for which its distance between $x$ and the distance between $y$ are unequal. All of these vertices must be moved by $\phi$, so there are at least $\frac{n}{4}(1 - \epsilon) \geq cn$ nonfixed points of $\phi$. This implies that there are at least $$cn(n-k) \geq c'n^2$$ nonfixed points of $\phi'$.

This in turn implies that there are at most $$r := k(n-k) - \frac{c'}{2}n^2$$ orbits of $\phi'$.

The edges within a single orbit of $\phi'$ must be all $1$ or all $3$. Thus there are at most $2^r$ metric spaces in $\secondaclass$ with parts $V_1,V_2$ which admit $\phi$ as an automorphism. As $k \neq n-k$, there are at most $k! (n-k)!$ nontrivial automorphisms on $A$.

%*exactly that many possible $\phi$.*

Combining these two bounds, and recalling that for a fixed $n,k$ there are $2^{k(n-k)}$ spaces in $\secondaclass$ with parts of size $k,n-k$, we get that for a fixed $n,k$, $k \neq n-k$, $k = \frac{n}{2}(1-\epsilon)$, the probability of a metric space in $\secondaclass$ with parts of size $k$ and $n-k$ respectively admitting a nontrivial automorphism is
\begin{align*}
\frac{k!(n-k)!2^r}{2^{k(n-k)}} .
%&= \frac{k!(n-k)!2^{k(n-k)-\frac{c'}{2}n^2}}{2^{k(n-k)}}\\
%&\leq \frac{n! 2^{k(n-k)-\frac{c'}{2}n^2}}{2^{k(n-k)}}\\
%&\leq \frac{n^n 2^{k(n-k)-\frac{c'}{2}n^2}}{2^{k(n-k)}}\\
%&=\frac{n^n}{2^{\frac{c'}{2}n^2}}
\end{align*}
%which goes to $0$ as $n$ goes to infinity.

We expand our argument now to consider all $k \leq \lfloor \frac{n}{2} \rfloor$.

%If we no longer fix $k$, that is, if we consider the probability that a random metric space in $\secondaclass$ on $n$ vertices admits a nontrivial automorphism, we must consider how many metric spaces with parts of size $k$ and $n-k$, for $k < \frac{n}{2}(1-\epsilon)$, admit a nontrivial automorphism $\phi$. 
The contribution to our calculation from small $k$, i.e. $k < \frac{n}{2}(1-\epsilon)$, will be much smaller than that from relatively large $k$, and therefore will not affect our limit. Thus for small $k$ we use the largest possible upper bound, that is, we say that there are no more than $2^{k(n-k)}$ metric spaces which admit a fixed nontrivial automorphism $\phi$.

This gives us that the probability that a random metric space in $\secondaclass$ on $n$ vertices admits a nontrivial automorphism is
$$\frac{\sum\limits_{k < \frac{n}{2}(1-\epsilon)} 2^{k(n-k)} + {\sum\limits_{\substack{k = \frac{n}{2}(1-\epsilon)\\k\neq \frac{n}{2}}} k!(n-k)!2^r} + (\frac{n}{2}!)^2 \psi(\phi,n,\frac{n}{2})}{\sum\limits_{k=0}\limits^{\lfloor \frac{n}{2}\rfloor}2^{k(n-k)}} $$
%$$\sum_{k \neq \frac{n}{2}(1-\epsilon)} \frac{k!(n-k)!2^r}{2^{k(n-k)}} + \sum_{k = \frac{n}{2}(1-\epsilon)} \frac{k!(n-k)!2^r}{2^{k(n-k)}},$$
where the last term in the numerator corresponds to the case $k = \frac{n}{2}$. In this term, $\psi(\phi,n,\frac{n}{2})$ is an upper bound on the number of spaces on $n = 2k$ vertices in $\secondaclass$ which are fixed by a single automorphism $\phi$. Thus, we consider now three separate cases:
\begin{enumerate}
    \item\label{item:smallk} $k < \frac{n}{2}(1-\epsilon)$;
    \item\label{item:mediumishk} $k = \frac{n}{2}(1-\epsilon), k \neq \frac{n}{2}$;
    \item\label{item:exactlymediumk} $k = \frac{n}{2}$.
\end{enumerate}

We separate the three terms in the numerator, and show that each limit goes to zero.

\textit{Term} (\ref{item:smallk})
\begin{equation*}
    \frac{\sum\limits_{k< \frac{n}{2}(1-\epsilon)} 2^{k(n-k)}}{\sum\limits_{k=0}^{\lfloor \frac{n}{2}\rfloor} 2^{k(n-k)}}
\end{equation*}
Each term in the numerator is bounded above by $$2^{[\frac{n}{2}(1-\epsilon)][n-(\frac{n}{2}(1-\epsilon))]}$$ and therefore the numerator is bounded by $$\frac{n}{2} 2^{\frac{n^2}{4}-\frac{n^2}{4}\epsilon^2}.$$
If we only keep the largest term in the denominator, we bound it below and therefore get a lower bound on the denominator $$2^{\lfloor \frac{n}{2}\rfloor \lceil \frac{n}{2} \rceil} > 2^{\frac{n^2}{4}-1}.$$

Therefore Term (\ref{item:smallk}) is bounded above by $$\frac{\frac{n}{2} 2^{\frac{n^2}{4}(1-\epsilon^2)}}{2^{\frac{n^2}{4}-1}} = \frac{n}{2^{\epsilon^2 \frac{n^2}{4}}}$$
which goes to $0$ as $n$ goes to infinity.\\

\textit{Term} $(\ref{item:mediumishk})$:
$$\frac{\sum\limits_{k = \frac{n}{2}(1-\epsilon)} k!(n-k)!2^r}{\sum\limits_{k=0}^{\lfloor \frac{n}{2} \rfloor} 2^{k(n-k)}}$$
For any $k,n$, we have that $k!(n-k)! \leq n^n$, and therefore $$\sum\limits_{k = \frac{n}{2}(1-\epsilon)} k!(n-k)! 2^r \leq n^n \sum\limits_{k = \frac{n}{2}(1-\epsilon)}2^r.$$ 
Recall that $r = \frac{n^2}{4}-c'n^2$, where $c'$ may depend on $k,n$. For a fixed $n$, let $\tilde{c} = \min\limits_{k \in \frac{n}{2}(1-\epsilon)}c'.$ Then 
$$\sum\limits_{k = \frac{n}{2}(1-\epsilon)}2^r \leq \sum\limits_{k=\frac{n}{2}(1-\epsilon)}2^{\frac{n^2}{4}-\tilde{c}n^2} = \frac{n}{2}2^{\frac{n^2}{4}-\tilde{c}n^2}.$$
Thus Term $(\ref{item:mediumishk})$ is bounded above by
$$\frac{n^n \frac{n}{2} 2^{\frac{n^2}{4}-\tilde{c}n^2}}{\sum\limits_{k=0}^{\lfloor \frac{n}{2} \rfloor} 2^{k(n-k)}}.$$
Again, we can bound the denominator below by $2^{\frac{n^2}{4}-1}$ and therefore Term $(2)$ is no more than $$\frac{n^{n+1}}{2^{\tilde{c}n^2}}$$ which goes to $0$ as $n$ goes to $\infty$.\\

\textit{Term} $(\ref{item:exactlymediumk})$:
$$\frac{(\frac{n}{2}!)^2 \psi(\phi,n,\frac{n}{2})}{\sum\limits_{k=0}^{ \frac{n}{2}} 2^{k(n-k)}}$$
We consider the fixed nontrivial automorphism $\phi$ acting on a space $A \in \mathcal{A}$ on $n$ vertices with parts $V_1,V_2$ where $|V_1| = k = \frac{n}{2}$ and $|V_2| = n-k = \frac{n}{2}$.

If $\phi$ switches the two parts $V_1,V_2$, then as with the argument for Term (\ref{item:mediumishk}) there are at most $2^{\frac{n^2}{4}-c n^2}$, for some constant $c$, metric spaces in $\secondaclass$ which admit $\phi$.

%We implicitly showed in the proof for Term (\ref{item:mediumishk}) that the percentage of metric spaces on $n$ vertices in $\mathcal{A}$ which admit a nontrivial automorphism which fixes the two parts is vanishingly small.

If $\phi$ switches the two parts $V_1,V_2$, then let $x \in A$ be such that $\phi(x) = y$ for some $y$ not in the same part as $x$.

We again define $V'$ to be all the edges between parts, and we define $\phi':V' \rightarrow V'$ by $\phi'(\{v_1,v_2\}) = \{\phi(v_1),\phi(v_2)\}$.

There is at most one edge containing $x$ which is fixed by $\phi'$, namely $\{x,\phi(x)\}$. Thus there are at most $\frac{n}{2}$ fixed edges between $V_1$ and $V_2$, Since there are $(\frac{n}{2})^2$ total edges between $V_1$ and $V_2$, then there are at least $\frac{n^2}{4} - \frac{n}{2}$ edges which are not fixed by $\phi'$.

Therefore there are at most
$$\frac{n^2}{4} - \frac{1}{2}\Big(\frac{n^2}{4}-\frac{n}{2}\Big) = \frac{n^2}{8} + \frac{n}{4} =: r'$$ orbits of $\phi'$.

This gives us that there are at most $2^{r'}$ such metric spaces $A$ which admit $\phi$ as an automorphism.

As we can again bound the denominator of term $(\ref{item:exactlymediumk})$ below by $2^{\frac{n^2}{4}-1}$, we have an upper bound for term $(\ref{item:exactlymediumk})$ of
\begin{align*}
    \frac{k^2 [2^{\frac{n^2}{4}-c n^2} + 2^{r'}]}{2^{\frac{n^2}{4}-1}} &= \frac{\frac{n^2}{4}[ 2^{\frac{n^2}{4}-c n^2} + 2^{r'}]}{2^{\frac{n^2}{4}-1}}\\
    &= \frac{\frac{n^2}{4} [2^{\frac{n^2}{4}-c n^2} + 2^{\frac{n^2}{8}+\frac{n}{4}}]}{2^{\frac{n^2}{4}-1}}\\
    &= \frac{n^2}{2^{c n^2 +1}} + \frac{n^2 2^{\frac{n}{4}}}{2^{\frac{n^2}{8}+1}}.
    %&= \frac{ n^2 }{2^{\frac{n^2}{8}-\frac{n}{4}+1}}.
\end{align*}
This again goes to $0$ as $n$ goes to $\infty$.\\

We have therefore shown that the probability of a random metric space in $\secondaclass$ on $n$ vertices admits a nontrivial automorphism asymptotically goes to $0$. Hence almost all metric spaces in $\secondaclass$ are asymmetric.
\end{proof}

Thus we have the following.

\begin{corlemma}\label{cor:asymmetricimplies01}
The amalgamation class $\secondaclass$ satisfies a labeled 0--1 law.
\end{corlemma}

\begin{proof}
This follows immediately from the fact that $\secondaclass$ satisfies an unlabeled 0--1 law and that almost all spaces in $\secondaclass$ are asymmetric.
\end{proof}

\section{Comparing almost sure theories and generic theories}\label{sec:asNgen}

We now have all the necessary ingredients to prove our last two main theorems.

\begin{proof}[Proof of Theorem \ref{thm:firstASnGen}]
%Lemmas \ref{lemma:firstclass01} and \ref{obs:AS1unique}, together with the observation that there are no finite models of $\Phi$, yield that $\firstaclass$ satisfies a 0--1 law. 
Theorem \ref{thm:firstUnlabeled01} and Lemma \ref{lemma:firstLabeled01} show that $\firstaclass$ satisfies both an unlabeled and a labeled 0--1 law. 

Since the generic theory of $\firstGamma$ contains the sentence
\begin{align*}
    \forall u \exists v \ni d(u,v) = 3
\end{align*}
and the almost sure theory does not, we have that these two theories are not the same.
\end{proof}

\begin{proof}[Proof of Theorem \ref{thm:secondASnGen}]
%Lemmas \ref{lemma:secondclass01} and \ref{lemma:secondASisCat}, together with the observation that $\Phi$ has no finite models, yields that $\secondaclass$ satisfies a 0--1 law. 
%
Theorem \ref{thm:UnlabeledSecond01} and Corollary \ref{cor:asymmetricimplies01}   show that $\secondaclass$ satisfies a 0--1 law. 

It follows immediately upon observation that $\secondGamma$ satisfies $\Phi$, the axiomatizing set for the almost sure theory of $\secondaclass$.
%Observation \ref{obs:secondGammaSatisfiesAS} gives us that 
Hence the almost sure theory of $\secondaclass$ matches the generic theory of $\secondGamma$.
\end{proof}

\section{Further Work}\label{sec:furtherwork}
\begin{enumerate}
    \item Do this analysis for more the finite diameter bipartite metrically homogeneous graphs of generic type. Are there necessary and sufficient conditions for determining when the almost sure theory will diverge from the generic theory?%Are all these known? Are there always 2?
    \item Establish 0--1 laws for the other metrically homogeneous graphs of generic type of diameter $3$. Compare the almost sure theories and generic theories and again try to find necessary and sufficient conditions for when the almost sure theories and the generic theories will diverge.\\
    \emph{We note here that the bipartite graphs are somewhat exceptional, and that we may reasonably expect different behavior from other metrically homogeneous graphs of generic type.}
   % \item Compare the almost sure theory and the generic theory of other metrically homogeneous graphs of diameter $3$ and again try to find necessary and sufficient conditions for when the almost sure theories and the generic theories will diverge.
    %Other frse limit graphs which are defined by forbidden configurations?
\end{enumerate}

\section{Acknowledgements}
Sincere thanks to Nathan Fox, Sam Braunfeld, Rehana Patel, and Alex Kruckman for their feedback and assistance.

\bibliographystyle{amsalpha}
\bibliography{References}\label{sec:biblio}

\end{document}